\documentclass{scrartcl}

\usepackage[british]{babel}
\usepackage[symbol,perpage]{footmisc}
\setfnsymbol{wiley}

\usepackage{graphicx}

\usepackage{amsmath,amssymb,amsthm}
\theoremstyle{plain}
\newtheorem{theorem}{Theorem}
\newtheorem{lemma}[theorem]{Lemma}
\newtheorem{prop}[theorem]{Proposition}
\newtheorem{cor}[theorem]{Corollary}

\theoremstyle{definition}
\newtheorem*{exa}{Example}

\def\C{\mathcal{C}}
\def\F{\mathcal{F}}
\def\P{\mathcal{P}}
\def\T{\mathcal{T}}
\def\U{\mathcal{U}}

\let\notto\nrightarrow

\DeclareMathOperator{\Sproink}{Sproink}
\def\PsiS{\mathop{\Psi\text{-Sproink}}}
\def\dpc{\delta_\pi\C}
\let\restrict\upharpoonright

\def\bs#1{\underline{#1}}

\begin{document}

\title{Adjoint functors and tree duality}
\author{Jan Foniok%
\\
\small ETH Zurich, Institute for Operations Research\\[-5pt]
\small R\"amistrasse 101, 8092 Zurich, Switzerland\\[-5pt]
\small\texttt{foniok@math.ethz.ch}%
\and
Claude Tardif%
\\
\small Royal Military College of Canada\\[-5pt]
\small PO Box 17000, Stn Forces, Kingston, Ontario\\[-5pt]
\small Canada, K7K 7B4\\[-5pt]
\small\texttt{Claude.Tardif@rmc.ca}
}
\date{6 May 2009}

\maketitle

\begin{abstract}
A family~$\T$ of digraphs is a \emph{complete set of obstructions}
for a digraph~$H$ if for an arbitrary digraph~$G$ the existence of
a homomorphism from~$G$ to~$H$ is equivalent to the non-existence
of a homomorphism from any member of~$\T$ to~$G$.  A digraph~$H$ is
said to have \emph{tree duality} if there exists a complete set of
obstructions~$\T$ consisting of orientations of trees. We show that
if $H$~has tree duality, then its arc graph~$\delta H$ also has tree
duality, and we derive a family of tree obstructions for~$\delta H$
from the obstructions for~$H$.

Furthermore we generalise our result to right adjoint functors on
categories of relational structures. We show that these functors
always preserve tree duality, as well as polynomial CSPs and the 
existence of near-unanimity functions.

\bigskip


\end{abstract}


\section{Introduction}

Our primary motivation is the \emph{$H$-colouring problem} (which has become
popular under the name \emph{Constraint Satisfaction Problem---CSP}):
for a fixed digraph~$H$ (a \emph{template}) decide whether an input
digraph~$G$ admits a homomorphism to~$H$. The computational complexity
of $H$-colouring depends on the template~$H$. For some templates the problem is
known to be NP-complete, for others it is tractable (a polynomial-time
algorithm exists). Assuming that $\mathrm{P}\ne\mathrm{NP}$, infinitely
many complexity classes lie strictly between P and NP~\cite{Lad:StruPTR},
but it has been conjectured that $H$-colouring belongs to no such
intermediate class for any template~$H$~\cite{FedVar:SNP}. This conjecture
has indeed been proved for symmetric templates~$H$~\cite{HelNes:Dicho}.

In this paper the focus is on tractable cases. Several conditions
are known to imply the existence of a polynomial-time algorithm for
$H$-colouring (definitions follow in the next two paragraphs): it
is the case if $H$~has a near-unanimity function (nuf), if $H$~has
bounded-treewidth duality, if $H$~has tree duality, if $H$~has finite
duality (see~\cite{CohJea:CCL,FedVar:SNP,HelNesZhu:DualPoly}). Some of
the conditions are depicted in the diagram (Fig.~\ref{fig:diag}).

\begin{figure}
\begin{center}
\includegraphics{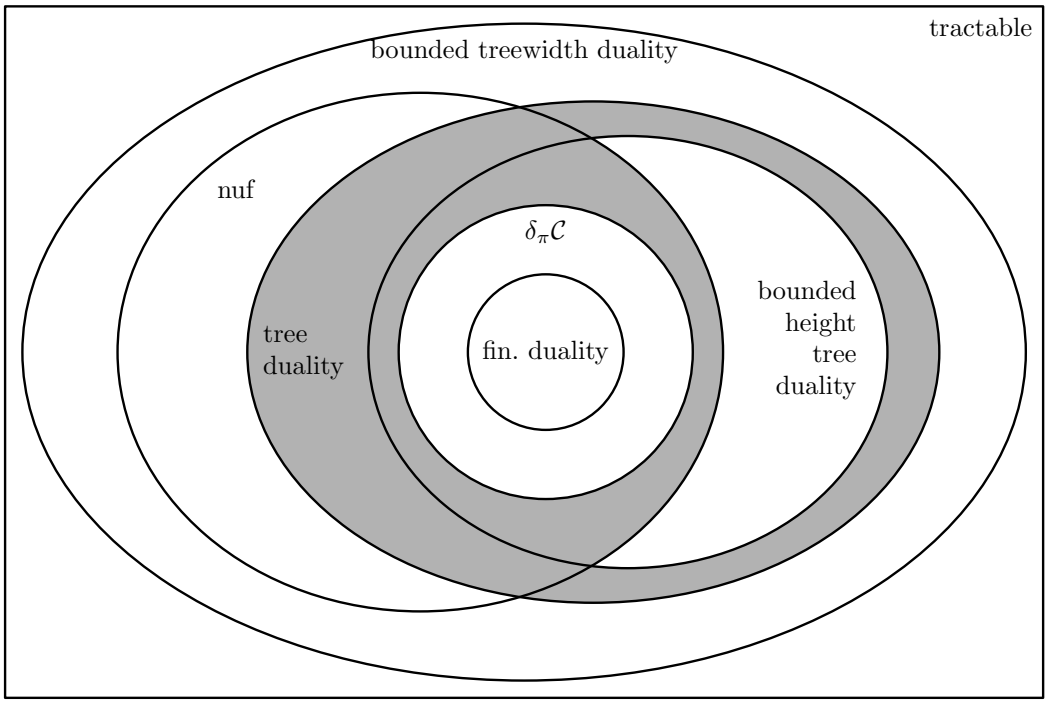}
\end{center}
\caption{The structure of tractable templates}
\label{fig:diag}
\end{figure}

A \emph{near-unanimity function} is a homomorphism~$f$
from~$H^k$ to~$H$ with $k\ge3$ such that for all $x,y\in V(H)$
we have $f(x,x,x,\dotsc,x)=f(y,x,x,\dotsc,x)=f(x,y,x,\dotsc,x)=
\dotsb=f(x,x,x,\dotsc,y)=x$. The power $H^k$ is the $k$-fold product
$H\times H\times\dotsb\times H$ in the category of digraphs and homomorphisms,
see~\cite{HelNes:GrH}.

A digraph is a tree (has treewidth~$k$) if its underlying undirected
graph is a tree (has treewidth~$k$, respectively).  A set $\F$ of
digraphs is a \emph{complete set of obstructions} for~$H$ if for an
arbitrary digraph~$G$ there exists a homomorphism from~$G$ to~$H$ if
and only if no $F\in\F$ admits a homomorphism to~$G$. A template has
\emph{bounded-treewidth duality} if it has a complete set of obstructions
with treewidth bounded by a constant; it has \emph{tree duality} if
it has a complete set of obstructions consisting of trees; and it has
\emph{finite duality} if it has a finite complete set of obstructions.

There is a fairly straightforward 
way to generate templates with finite duality. 
For an arbitrary tree~$T$ there exists a  digraph~$D(T)$ such that $\{T\}$
is a complete set of obstructions for~$D(T)$. The digraph~$D(T)$
is unique up to homomorphic equivalence%
\footnote{Two digraphs $H$ and $H'$ are \emph{homomorphically equivalent}
if there exists a homomorphism from~$H$ to~$H'$ as well as a homomorphism
from~$H'$ to~$H$. Clearly, if $H$ and~$H'$ are homomorphically equivalent,
then $H$-colouring and $H'$-colouring are equivalent problems, because
$H$ and $H'$ admit homomorphisms from exactly the same digraphs.}%
; it is called the \emph{dual} of~$T$. Several explicit constructions are
known (see~\cite{F:Diss,Kom:Phd,NesTar:Dual,NesTar:Short}).
If $\F$ is a finite set of oriented trees, then the product
$D=\prod_{T\in\F}D(T)$ is a template with finite duality and $\F$~is a
complete set of obstructions for~$D$. This construction yields all
digraphs with finite duality~\cite{NesTar:Dual}, thus also proving that
finite duality implies tree duality.

Encouraged by the full description of finite dualities, we aim to provide
a construction for some more digraphs with tree duality. To this end
we use the arc-graph construction and consider the
class~$\delta_\pi{\cal C}$ of digraphs generated from finite duals
by taking iterated arc graphs and finite Cartesian products. We show that
all templates in this class have tree duality. We provide an explicit
construction of the resulting tree obstructions, which allows us to show
that all the digraphs in~$\dpc$ have in fact \emph{bounded-height tree
duality}, that is, they have a complete set of obstructions consisting
of trees of bounded algebraic height (these are tree obstructions that
allow a homomorphism to a fixed directed path).
In this context we also prove that the problem of existence of a complete
set of obstructions consisting of trees with bounded algebraic height
is decidable.

The arc-graph construction is a special case of a more general phenomenon:
it is a right adjoint in the category of digraphs and homomorphisms. We
show in the more general setting of the category of relational structures
that right adjoints (characterised by Pultr~\cite{Pul:The-right-adjoints}
for all locally presentable categories) preserve tractability of templates
and moreover they preserve tree duality and existence of a near-unanimity
function. In this case, nevertheless, it remains open to provide a nice
general description of complete sets of obstructions.

We use some notions and properties of graphs and homomorphisms which the
reader can look up in~\cite{HelNes:GrH}, as well as some category-theory
notions, for which, e.g.~\cite{BarWel:Cat,Mac:Cat} may be consulted.

\section{Arc graphs and tree duality}

Let $G=(V,A)$ be a digraph. The \emph{arc graph} of~$G$ is the digraph
$\delta G =(A,\delta A)$, where
\[\delta A=\bigl\{((u,v),(v,w)) : (u,v),(v,w)\in A\bigr\}.\]
Notice that $\delta$ is an endofunctor%
\footnote{An \emph{endofunctor} is a functor from a category to itself.}
in the category of digraphs
and homomorphisms.  This implies in particular that if $G\to H$, then
$\delta G \to \delta H$. (The notation $G\to H$ means
that there exists a homomorphism from~$G$ to~$H$.)

If $G$ is a digraph and $\sim$~is an equivalence relation on
its vertex set~$V(G)$, the \emph{quotient}~$G/{\sim}$ is the
digraph~$(V(G)/{\sim},A)$, where $V(G)/{\sim}$~is the set of all
equivalence classes of~$\sim$ on~$V(G)$, and for $X,Y\in V(G)/{\sim}$
we have $(X,Y)\in A$ if and only if there exist $x\in X$ and $y\in Y$
such that $(x,y)\in A(G)$.

Suppose still that $G=(V,A)$ is a digraph. Let $V'=\{o_u,t_u:u\in V\}$
and let $A'=\{(o_u,t_u) : u\in V\}$. Define the relation~$\sim_0$
such that $t_u\sim_0 o_v$ if and only if $(u,v)\in A$. Let $\sim$
be the minimal equivalence relation on~$V'$ containing~$\sim_0$. Set
$\delta^{-1}G=(V',A')/{\sim}$. In the following, we use the notation
$V'(G)=V'$, $A'(G)=A'$ and $\sim_0$ and $\sim$ for the sets and relations
appearing in the definition of~$\delta^{-1}$; the precise meaning will
be clear from the context.
Now $\delta^{-1}$ is also an endofunctor in the category of digraphs.
Strictly speaking, it is not an inverse of $\delta$;
its name is chosen because of the following property.

\begin{prop}
\label{prop:delta}
For any digraphs $G$ and $H$,
\[G\to\delta H \qquad\text{if and only if}\qquad \delta^{-1}G\to H.\]
\end{prop}

\begin{proof}
Let $f: G \to \delta H$ be a homomorphism.  Then there exist two
homomorphisms $o, t: G \to H$ such that $f(u) = (o(u), t(u))$ for all $u
\in V(G)$.  Define the mapping $\hat g : V'(G) \to V(H)$ by $\hat g(o_u)
= o(u)$ and $\hat g (t_u) = t(u)$.  If $t_u \sim_0 o_v$, then $(u,v)
\in A(G)$, whence $(f(u),f(v)) \in A(\delta H)$ and thus $t(u) = o(v)$.
Therefore $\hat g$~is constant on the equivalence classes of~$\sim$, and it
induces a homomorphism from ${A'(G)/{\sim}} = \delta^{-1}G$ to~$H$.

Conversely, let $g: \delta^{-1}G \to H$ be a homomorphism.
We define $f: V(G) \to V(\delta H)$ by 
$f(u) = (g(o_u/{\sim}), g(t_u/{\sim}))$.
If $(u,v) \in A(G)$, then  $t_u/{\sim} = o_v/{\sim}$,
whence $(f(u), f(v)) \in A(\delta H)$. Therefore
$f$~is a homomorphism.
\end{proof}

Thus $\delta$ and $\delta^{-1}$ are Galois adjoints%
\footnote{Let $X$ and $Y$ be partially ordered sets. Mappings
$\phi:X\to Y$ and $\psi:Y\to X$ are \emph{Galois adjoints} if $\phi(x)\le_{Y}y
\,\Leftrightarrow\, x\le_X \psi(y)$ for all elements $x\in X$ and $y\in Y$.}
with respect to
the ordering by existence of homomorphisms. They are in fact adjoint
functors in the category of digraphs and homomorphisms. We return to this
topic in Section~\ref{sec:functors}. For the moment we aim to prove 
that $\delta$ preserves tree duality. More precisely, from the 
family~${\cal T}$ of tree obstructions of~$H$, we will derive the 
family $\Sproink({\cal T})$ of tree obstructions of~$\delta H$.

The \emph{algebraic height} of an oriented tree~$T$ is the minimum number of
arcs of a directed path to which $T$~maps homomorphically.  
The algebraic height of every finite oriented tree is well-defined and finite, 
since every such tree admits a homomorphism to some finite directed path. 
Thus a tree~$T$ is \emph{of height at most one} if its vertex set 
can be split into two parts $0_T, 1_T$ in such a way that for every 
arc~$(x,y)$ of~$T$ we have $x \in 0_T$ and $y \in 1_T$.  
Note that if the tree~$T$ has no arcs, then it has only one vertex
and thus one of the sets $0_T$, $1_T$ is empty and the other one is
a singleton.

Let $T$ be a tree. For every vertex~$u$ of~$T$, let $F(u)$ be a tree
of height at most one. For each arc~$e$ of~$T$ incident with~$u$, let
there be a fixed vertex $v(e,F(u))$ in~$F(u)$ such that if $u$~is the
initial vertex of~$e$, then $v(e,F(u))\in 1_{F(u)}$, and if $u$~is the
terminal vertex of~$e$, then $v(e,F(u))\in 0_{F(u)}$.%
\footnote{It follows that if $u$~is neither a source nor a sink of~$T$,
then both $0_{F(u)}$ and $1_{F(u)}$ are non-empty, and so in this case
$F(u)$~is not a single vertex. If $u$ is a source or a sink of~$T$,
then $F(u)$ may be an arbitrary tree of height at most one.}
A tree~$S$ is now constructed by taking all the trees~$F(u)$ for all
vertices~$u$ of~$T$, and by identifying the vertex $v(e,F(u))$ with
$v(e,F(u'))$ whenever $e=(u,u')$ is an arc of~$T$.

Any such tree~$S$ constructed from~$T$ by the above procedure is called
a \emph{sproink} of~$T$. The set of all sproinks of a tree~$T$ is denoted
by $\Sproink(T)$. The following lemma asserts that sproinks 
of obstructions for a template~$H$ are indeed obstructions 
for its arc graph~$\delta H$.

\begin{lemma}
\label{lem:sproink}
Let $T$ be a tree and $H$ a digraph such that $T\notto H$.
If $S\in\Sproink(T)$, then $S\notto \delta H$.
\end{lemma}

\begin{proof}
We prove that $T\to\delta^{-1}S$. Consequently $\delta^{-1}S\notto H$ 
because $T\notto H$, and therefore $S\notto\delta H$ by
Proposition~\ref{prop:delta}.

Thus let $S\in\Sproink(T)$. For a vertex~$u$ of~$T$, consider the
tree~$F(u)$, which is a subgraph of~$S$. Since $F(u)$~has height
at most one, its vertices are partitioned into the sets $0_{F(u)}$
and~$1_{F(u)}$. The set~$V'(S)$, which appears in the definition
of~$\delta^{-1}S$, contains~$V'(F(u))$ as a subset. If $(x,y)$ is an
arc of~$F(u)$, then $t_x\sim_0 o_y$. Thus whenever $x\in0_{F(u)}$ and
$y\in 1_{F(u)}$, then $t_x\sim o_y$. Hence for any vertex~$u$ of~$T$
there exists a unique vertex~$f(u)$ of~$\delta^{-1}S$ that is equal
to $t_x/{\sim}$ for all $x\in 0_{F(u)}$ and to $o_y/{\sim}$ for all
$y\in 1_{F(u)}$.

In this way, we have defined a mapping $f: V(T) \to V(\delta^{-1}S)$.

Now assume that $e=(u,v)$ is an arbitrary arc of $T$. Then the vertex
$v(e,F(u))$, which belongs to~$1_{F(u)}$, has been identified with
$v(e,F(v))$, which belongs to~$0_{F(v)}$. Let this identified vertex
be~$x$; it is a vertex of~$S$. By definition, $f(u)=o_x/{\sim}$ because
$x\in1_{F(u)}$, and $f(v)=t_x/{\sim}$ because $x\in 0_{F(v)}$. Of
course $(o_x/{\sim},t_x/{\sim})\in A(\delta^{-1}S)$. Therefore $f:
T\to\delta^{-1}S$ is a homomorphism, as we have promised to prove.
\end{proof}

For a set~$\F$ of trees, let $\Sproink(\F)= \bigcup_{T\in\F}\Sproink(T)$.

\begin{theorem}
\label{thm:spr-thm}
Let $\F$ be a set of trees which is a complete set of obstructions for
a template~$H$. Then $\Sproink(\F)$ is a complete set of obstructions
for~$\delta H$.
\end{theorem}

\begin{proof}
Lemma~\ref{lem:sproink} implies that $\Sproink(\F)$ is a set of
obstructions for~$\delta H$. It remains to prove that it is complete, that
is whenever $G\notto\delta H$, then there exists some $S\in\Sproink(\F)$
such that $S\to G$.

So let $G\notto\delta H$. Thus by Proposition~\ref{prop:delta} we have
$\delta^{-1}G\notto H$.  Hence there exists a tree $T\in\F$ such that
$T\to \delta^{-1}G$, because $\F$~is a complete set of obstructions
for~$H$. Consequently it suffices to prove that if $T\to\delta^{-1}G$
then there exists $S\in\Sproink(T)$ such that $S\to G$.

Thus assume that $f:T\to\delta^{-1}G$ is a homomorphism. For every $u
\in V(T)$, the image~$f(u)$ is a $\sim$-equivalence class; put
\begin{align*}
1_{u} & = \{ y \in V(G) : o_y \in f(u)\},\\
0_{u} & = \{ x \in V(G) : t_x \in f(u)\}.
\end{align*}

Then $f(u) = 1_{u} \cup 0_u$, and by the definition of~$\sim$ as the least
equivalence containing~$\sim_0$, there exists a tree~$F(u)$ of height at
most one and a homomorphism $g_u: F(u) \to G$ such that $g_u(0_{F(u)})
= 0_{u}$ and  $g_u(1_{F(u)}) = 1_{u}$.  For every arc~$(u,v)$ of~$T$,
we have $(f(u), f(v)) \in A(\delta^{-1}G)$ so there exists $x \in V(G)$
such that $o_x \in f(u)$ and $t_x \in f(v)$.

We then select $y \in 1_{F(u)}$ and $z \in 0_{F(v)}$ such that
$g_u(y) = g_v(z) = x$, and identify them.  Proceeding with all such
identifications, we construct a tree $S \in \Sproink(T)$ such that
$g = \bigcup_{u \in V(T)} g_u : S \to G$ is a well-defined homomorphism.
\end{proof}

\begin{cor}
\label{cor:arc}
If a digraph $H$ has tree duality, then its arc graph~$\delta H$ also
has tree duality.
\qed
\end{cor}

\begin{exa}
Consider $T=\vec P_4$, the directed path with four arcs, and its
dual~$D=\vec T_4$, the transitive tournament on four vertices. Here
$\delta D$~has six vertices, but its core%
\footnote{The {\em core} of a digraph is any of its smallest subgraphs to
which it admits a homomorphism. Every digraph~$H$ has a unique core~$C$
(up to isomorphism), which is moreover the only core homomorphically equivalent to it.
In fact, the core~$C$ of~$H$ is a \emph{retract} of~$H$, which means that
there exists a homomorphism $\rho:H\to C$ whose restriction on~$C$ is
the identity mapping (such a homomorphism is called a \emph{retraction}).}
is the directed path $\vec P_2$ with two arcs.  It is well known
that a directed graph $G$ admits a homomorphism to $\vec P_2$ if and
only if it does not admit a homomorphism from a ``thunderbolt'',
that is, an oriented path with two forward arcs at the beginning
and at the end, and with an odd-length alternating path between them
(see Fig.~\ref{fig:thunder}). Thus the family of all thunderbolts is a
complete set of tree obstructions for $\vec P_2$.

\begin{figure}
\begin{center}
\includegraphics{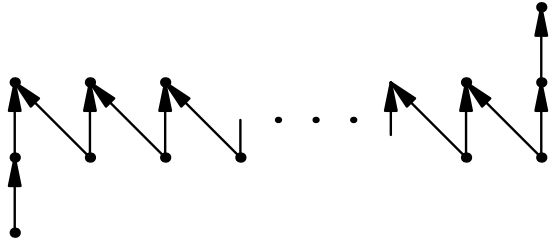}
\end{center}
\caption{A thunderbolt}
\label{fig:thunder}
\end{figure}

Our construction $\Sproink(T)$ gives all obstructions obtained by stacking
five trees $L_0$, $L_1$, $L_2$, $L_3$,~$L_4$ of height at most one, with
one top vertex of~$L_i$ identified with one bottom vertex of~$L_{i+1}$ for
$i = 0, 1, 2, 3$.  The example of thunderbolts shows that in fact $L_0$
can be restricted to a single (top) vertex, and $L_4$~can be restricted
to a single (bottom) vertex. The same holds for leaves of
general trees. Also,  $L_1$, $L_2$, $L_3$ can be restricted to \emph{paths}
of height one, and it is also true in general that it is sufficient to
consider sproinks obtained by replacing vertices by \emph{paths} of height at
most one. In fact the name ``sproink'' is inspired by picturing such a
path springing out of every non-leaf of~$T$.
\end{exa}

The results of this section show that we can construct an interesting
class of templates with tree duality by repeatedly applying the arc-graph
construction to digraphs with finite duality.  Moreover, if templates
$H_1$, $H_2$,~\dots,~$H_k$ all have tree duality, then also their product
$H_1\times H_2\times\dotsb\times H_k$ has tree duality as the union of
the respective complete sets of obstructions of the factors is a complete
set of obstructions for the product.  The resulting class of templates
is subject to examination in the next section.

\section{Finite duality} \label{sec:fd}

Following \cite{NesTar:Dual}, every tree $T$ admits a dual $D(T)$ such
that for every digraph~$G$, we have $G \rightarrow D(T)$ if and only
if $T \notto G$. A digraph~$H$ has finite duality if and only if it is
homomorphically equivalent to a finite product of duals of trees.

In this section, we consider the class $\dpc$, the smallest class of
digraphs that contains all duals of trees and is closed under taking
arc graphs, finite products and homomorphically equivalent digraphs.
It follows from Corollary~\ref{cor:arc} that all elements of~$\dpc$
have tree duality. Moreover we know how to construct a complete set
of obstructions for each of these templates, using iterated $\Sproink$
constructions and unions. The question then arises as to how significant
the class~$\dpc$ is within the class of digraphs with tree duality. It
turns out that the digraphs in~$\dpc$ have properties that are not shared
by all digraphs with tree duality.

A digraph~$H$ has \emph{bounded-height tree duality} provided there
exists a constant~$m$ such that $H$~admits a complete set of obstructions
consisting of trees of algebraic height at most~$m$.

\begin{prop}
\label{prop:dpc}
\begin{itemize}
\item[(i)] Every core in $\dpc$ admits a near-unanimity function.
\item[(ii)] Every member of $\dpc$ has bounded-height tree duality.
\end{itemize}
\end{prop}

\begin{proof} (i): By Corollary~4.5 of~\cite{Tar:FOD}, 
every structure with finite duality admits a near-unanimity function.
Therefore it suffices to show that the class of structures
admitting a near-unanimity function is closed under taking
cores, finite products and the arc-graph construction. 

Let $C$ be the core of~$H$, $\rho: H \to C$ a retraction 
and $f:H^k\to H$ a near-unanimity function. 
Since $C$~is an induced subgraph of~$H$, the restriction
$\rho \circ f\restrict C^k$ is a near-unanimity function on~$C$.

Suppose $f_i:H_i^{k_i}\to H_i, i = 1, \ldots, m$ are near-unanimity
functions. For $k = \max\{k_i : i = 1, \ldots, m\}$, we define
$k$-ary near-unanimity functions $g_i: H_i^k \to H_i$ by
$g_i(x_1, \ldots, x_k) = f_i(x_1, \ldots, x_{k_i})$.
For $H = \Pi_{i = 1}^m H_i$ we then define a near-unanimity
function $g: H^k \to H$ coordinate-wise, by putting
\[g((x_{1,1},\ldots, x_{m,1}), \ldots, (x_{1,k}, \ldots, x_{m,k}))
= (g_1(x_{1,1},\ldots,x_{1,k}), \ldots, g_m(x_{m,1},\ldots, x_{m,k})).\]

Now suppose that $f: H^k \to H$ is a near-unanimity function.  Then $\left
( \delta H \right )^k$ is naturally isomorphic to~$\delta(H^k)$, and we
define $g: \left ( \delta H \right )^k \to \delta H$ by \[g((u_1,v_1),
\ldots, (u_k,v_k)) = (f(u_1,\ldots, u_k),f(v_1,\ldots,v_k)).\] The fact
that $f$ is a homomorphism implies that $g$~is well defined, and $g$~is
a homomorphism by the definition of adjacency in~$\delta H$.  Also, $g$
clearly satisfies the near-unanimity identities, so it is a near-unanimity
function on~$\delta H$.

\bigskip

(ii): The class of digraphs with bounded-height tree duality is obviously
preserved by taking cores and finite products. By Theorem~\ref{thm:spr-thm}, if
$H$~has a complete set of obstructions consisting of trees of algebraic
height at most~$k$, then $\delta H$~has a complete set of obstructions
consisting of trees of algebraic height at most~${k+1}$, so the class
of digraphs with bounded-height tree duality is also preserved by the
arc-graph construction.
\end{proof}

We know a core digraph with tree duality but no near-unanimity function
and no bounded-height tree duality. (The example is complicated and out of
the scope of this paper, therefore we omit it.)  Thus the class~$\dpc$
does not capture all core digraphs with tree duality. The problem
of generating all structures with tree duality by means of suitable
functors applied to structures with finite duality remains nevertheless
interesting.

Membership in $\dpc$ is not known to be decidable. In the remainder
of this section, we show that bounded-height tree duality is decidable.

Given a digraph $H$, the \emph{$n$-th crushed cylinder}~$H_n^*$
is the quotient $(H^2\times P_n) / {\simeq_n}$, where
$P_n$~is the path with arcs 
$(0,0),(0,1), (1,2), \cdots, (n-1,n), (n,n)$,
and $\simeq_n$ is the equivalence defined by
\[
(u,v,i) \simeq_n (u',v',j) \Leftrightarrow
\begin{cases}
\text{$i = j = 0$ and $u = u'$}, \\
\text{or $i = j = n$ and $v = v'$}, \\
\text{or $(u,v,i) =(u',v',j)$}.
\end{cases}
\]

\begin{theorem}
\label{thm:123}
For a core digraph $H$ with tree duality, the following
are equivalent:
\begin{itemize}
\item[(1)] $H$ has bounded-height tree duality,
\item[(2)] For some $n$ we have $H_n^* \to H$. 
\item[(3)] There exists a directed (upward) path from
the first projection to the second in $H^{H^2}$.
\end{itemize}
\end{theorem}

\begin{proof}
{\bf (1) $\Rightarrow$ (2):}
The two subgraphs obtained from $H_n^*$ by removing 
the two ends both admit homomorphisms to $H$. Therefore, if
a tree obstruction of $H$ admits a homomorphism to
$H_n^*$, its image must intersect the two ends hence its
algebraic length must be at least $n$.

\textbf{$\neg{}$(1) $\Rightarrow$ $\neg{}$(2):} 
Let $T$ be a critical obstruction of~$H$ of algebraic length~$n+2$.
Let $T_0$, $T_n$ be the subgraphs of~$T$ obtained by removing
the vertices of height $0$ and~$n+2$ respectively. Then there exists
homomorphisms $f_0: T_0 \to H$ and $f_n: T_n \to H$.
Let $h: T \to P_{n+2}$ be the height function of~$T$. We define
a map $f: T \to H_n^*$ by
\[
f(u) =
\begin{cases}
(f_n(u), f_0(u), h(u)-1) / {\simeq_n} &
	\text{if $h(u) \not \in \{0, n+2\}$,}\\
(f_n(u), f_n(u), 0)/ {\simeq_n}&
	\text{if $h(u) = 0$,}\\
(f_0(u), f_0(u), n)/ {\simeq_n}&
	\text{if $h(u) = n+2$.}
\end{cases}
\]
Let $(u,v)$ be an arc of $T$. Then $h(v) = h(u) + 1$. If 
$\{h(u), h(v) \} \cap \{ 0, n+2 \} = \emptyset$, we clearly have
$(f(u), f(v)) \in A(H_n^*)$. If $h(u) = 0$, then
$f(u) = (f_n(u), f_n(u), 0)/ {\simeq_n}$ 
is an in-neighbour of $(f_n(v), f_n(v), 0)/ {\simeq_n}
= (f_n(v), f_0(v), 0)/ {\simeq_n} = f(v)$,
and if $h(v) = n+2$, 
$f(v) = (f_0(v), f_0(v), n)/ {\simeq_n}$ 
is an out-neighbour of~$f(u)$ because
\[(f_0(u), f_0(u), n)/ {\simeq_n}
= (f_n(u), f_0(u), n)/ {\simeq_n} = f(u).\]
Therefore $f$ is a homomorphism.

\textbf{(2) $\Leftrightarrow$ (3):}
This equivalence follows easily from the definition.
\end{proof}

\begin{cor}
The problem whether an input digraph has bounded-height tree
duality is decidable.
\end{cor}

\begin{proof}
It is decidable whether a digraph has tree duality~\cite{FedVar:SNP}
(see Theorem~\ref{thm:u} below).  For a digraph with tree duality, bounded
height of the obstructions (the condition (1) of Theorem~\ref{thm:123})
is equivalent to the condition~(3), which involves directed reachability
in a finite graph. Hence bounded-height tree duality is decidable.
\end{proof}

\section{Adjoint functors and generation of tractable templates}
\label{sec:functors}

The correspondence of Proposition \ref{prop:delta} can be
extended to a wide class of functors presented in this
section. To illustrate this extension, we first redefine
$\delta$ in terms of patterns. Let $P$ be the digraph 
with vertices $0, 1$ and arc $(0,1)$, and $Q$ the digraph with
vertices $0, 1, 2$ and arcs $(0,1), (1,2)$. Furthermore
let $q_1, q_2: P \to Q$ be the homomorphisms mapping
the arc $(0,1)$ to $(0,1)$ and $(1,2)$ respectively.
For an arbitrary digraph $G$, its arc graph~$\delta G$ can be described as
follows: The vertices of $\delta G$ are the arcs of $G$,
that is, the homomorphisms $f: P \to G$. The arcs
of $\delta G$ are the couples of consecutive arcs in $G$,
that is, the couples $(f_1,f_2)$ such that there exists
a homomorphism $g: Q \to G$ satisfying
$g \circ q_1 = f_1$ and $g \circ q_2 = f_2$.
Thus the functor $\delta$ is generated
by the pattern $\{P, (Q,q_1,q_2))\}$ in a way that generalises
quite naturally.

The rest of this section deals with relational structures.
A \emph{vocabulary} is a finite set $\sigma = \{ R_1 , \dotsc , R_m
\}$ of relation symbols, each with an arity~$r_i$ assigned to it. A
$\sigma$-structure is a relational structure $A = \langle \bs A; R_1
(A), \dotsc , R_m (A)\rangle$ where $\bs A$ is a non-empty set called
the \emph{universe} of~$A$, and $R_i (A)$~is an $r_i$-ary relation
on~$\bs A$ for each~$i$.  Homomorphisms of relational structures
are relation-preserving mappings between universes; a homomorphism
is defined only between structures with the same vocabulary. Cores,
trees, quotient structures, etc.\ can also be defined in the context
of relational structures, consult~\cite{LoTar:mosfd} for the details
(see also~\cite{Hod:A-shorter-model,Tar:FOD}).  The notions of the
constraint satisfaction problem, template, and tree duality also carry
over naturally from the setting of digraphs.

Let $\sigma$ and $\tau$ be two vocabularies. Let $P$ be a
$\sigma$-structure, and for every relation $R$ of $\tau$ of arity $r =
a(R)$, let $Q_R$ be a $\sigma$-structure with $r$ fixed homomorphisms
$q_{R,i}: P \to Q_R$ for $i = 1, \ldots, r$.
Then the family
$\{ P \} \cup \{ (Q_R, q_{R,1}, \ldots, q_{R,a(R)}) : R \in \tau \}$
defines a functor $\Psi$ from the category~$\mathcal{A}$
of $\sigma$-structures to the category~$\mathcal{B}$ of
$\tau$-structures as follows.
\begin{itemize}
\item For a $\sigma$-structure $A$, let $B = \Psi A$ be a
$\tau$-structure whose universe is the set of all 
homomorphisms $f: P \to A$.
\item For every relation $R$ of $\tau$ of arity $r = a(R)$,
let $R(B)$ be the set of $r$-tuples $(f_1, \ldots, f_r)$ such that
there exists a homomorphism $g: Q_{R} \to A$ such that for
$i = 1, \ldots, r$ we have $g \circ q_{R,i} = f_i$.
\end{itemize}

It was shown by Pultr~\cite{Pul:The-right-adjoints} that functors~$\Psi$
defined by means of patterns are right adjoints into a category of
relational structures characterised by axioms of a specific type. We
exhibit their corresponding left adjoints~$\Psi^{-1}$ in the case when
both the domain and the range of~$\Psi$ is the category of all relational
structures with a given vocabulary.

For every $\tau$-structure $B$, we define
a $\sigma$-structure $\Psi^{-1} B  = A/{\sim}$, where
\begin{itemize}
\item $A$ is a disjoint union of $\sigma$-structures;
for every element $x$ of the universe of $B$, $A$
contains a copy $P_x$ of $P$, and for every $R \in \tau$
and $(x_1, \ldots, x_r) \in R(B)$, $A$ contains 
a copy $Q_{R,(x_1, \ldots, x_r)}$ of $Q_R$.
\item $\sim$ is the least equivalence which identifies
every element $u$ of $P_{x_i}$ with its image $q_{R,i}(u)$
in $Q_{R,(x_1, \ldots, x_r)}$, for every $R \in \tau$,
every $(x_1, \ldots, x_r) \in R(B)$ and every 
$i \in \{1, \ldots, r\}$. 
\end{itemize}

\begin{prop}[\cite{Pul:The-right-adjoints}]
\label{prop:Psi}
For any $\tau$-structure $B$ and $\sigma$-structure $A$,
\[B\to\Psi A  \qquad\text{if and only if}\qquad \Psi^{-1} B \to A.\]
\end{prop}

\begin{proof}
Let $h: B\to\Psi A $ be a homomorphism, and put $h(b) = f_b: P\to A$.
Then for every $b \in B$, the mapping~$f_b$ corresponds to a well-defined 
homomorphism to~$A$ from a copy~$P_b$ of~$P$. Also, for every
$R \in \tau$ and $(b_1, \ldots, b_r) \in R(B)$, we have
$(h(b_1), \ldots, h(b_r)) \in R(\Psi A )$, so there exists a homomorphism
$g_{(b_1, \ldots, b_r)}: Q_R\to A$ such that 
$f_{b_i} = g_{(b_1, \ldots, b_r)} \circ q_{R,i}$ for $i = 1, \ldots, r$;
the mapping $g_{(b_1, \ldots, b_r)}$ corresponds to a well-defined 
homomorphism from a copy $Q_{R,(b_1, \ldots, b_r)}$ of~$Q_R$ to~$A$.
Therefore $\bigcup_{b \in B} f_b \cup 
\bigcup_{\tau} \bigcup_{R(B)} g_{(b_1, \ldots, b_r)}$ corresponds to a 
well-defined homomorphism $\hat{h}: \bigcup_{b \in B} P_b \cup 
\bigcup_{\tau} \bigcup_{R(B)} Q_{R,(b_1, \ldots, b_r)} \to A$,
such that if $x \sim y$, then $\hat{h}(x) = \hat{h}(y)$.
Therefore $\hat{h}$~induces a homomorphism from the quotient structure
$\Psi^{-1} B $ to~$A$.

Conversely, if $h: \Psi^{-1} B \to A$ is a homomorphism,
we define a homomorphism $\hat{h}: B\to\Psi A $
by $\hat{h}(b) = f_b$, where $f_b$~corresponds to the 
restriction of~$h$ to the quotient of~$P_b$ in~$\Psi^{-1} B $.
Indeed, if $R \in \tau$ and $(b_1, \ldots, b_r) \in R(B)$, then the
restriction of~$h$ to the quotient of 
$Q_{R,(b_1, \ldots b_r)}$ in $\Psi^{-1} B $ corresponds to
a homomorphism $g: Q_R \to A$ such that 
$f_{b_i} = g \circ q_{R,i}$ for $i = 1, \ldots, r$, whence
$(\hat{h}(b_1), \ldots, \hat{h}(b_r)) \in R(\Psi A )$.
\end{proof}

\begin{cor}
If a $\sigma$-structure $A$ has polynomial CSP, 
then the $\tau$-structure $\Psi A $ also
has polynomial CSP.
\qed
\end{cor}

In fact, Corollary~\ref{cor:arc} generalises as follows.

\begin{theorem} \label{thm:psi}
If a $\sigma$-structure $A$ has tree duality, 
then the $\tau$-structure $\Psi A $ also
has tree duality.
\end{theorem}

We prove Theorem~\ref{thm:psi} using Feder and Vardi's characterisation
of structures
with tree duality. For a $\sigma$-structure $A$, let
$\U A$ be the $\sigma$-structure defined as follows. The
universe of $\U A$ is the set of all nonempty subsets of 
$A$, and for $R \in \sigma$ of arity $r$,
$R(\U A)$ is the set of all $r$-tuples 
$(X_1, \ldots, X_{r})$ such that for all $j \in \{1, \ldots, r\}$ and
$x_j \in X_j$ there exist $x_k \in X_k, k \in \{1, \ldots,
r\}\setminus \{j\}$ such that $(x_1, \ldots, x_{r}) \in
R(A)$.

\begin{theorem}[\cite{FedVar:SNP}] \label{thm:u}
A structure $A$ has tree duality if and only if there exists
a homomorphism from $\U A$ to $A$.
\qed
\end{theorem}

\begin{proof}[Proof of Theorem~\ref{thm:psi}]
Suppose $A$ has tree duality. Then there is a homomorphism
$f:\U A\to A$. Let $U=\P(\Psi A )\setminus\{\emptyset\}$ be the 
universe of~$\U\Psi  A $ and let $S\in U$. 
For $p \in P$, define $S_p ={\{f(p): f \in S\}} \in \U A$,
and $f_S(p) = f(S_p)$. We claim that $f_S: P \to A$ is a 
homomorphism. Indeed, for $R \in \sigma$ and
$(p_1, \ldots, p_r) \in R(P)$, the $r$-tuples
$(f(p_1), \ldots, f(p_r)) \in R(A)$ for all $f \in S$ prove that
$(S_{p_1}, \ldots, S_{p_r}) \in R(\U A)$,
whence $(f_S(p_1), \ldots, f_S(p_r)) = (f(S_{p_1}), \ldots, f(S_{p_r}))
\in R(A)$.

Thus we define a map
$\hat{f}: \U \Psi A  \to \Psi A $ by $\hat{f}(S) = f_S$.
We show that it is a homomorphism. For $R \in \tau$ and
$(S_1, \ldots, S_r) \in R(\U \Psi A )$, every 
$f_i \in S_i$, $1 \leq i \leq r$ is contained in an $r$-tuple
$(h_1, \ldots, h_r) \in R(\Psi A )$ with $f_j \in S_j$ for $1 \leq j \leq r$
and $h_i = f_i$,
whence there exists a homomorphism 
$g_{(h_1, \ldots, h_r)}: Q_R \to A$ such that
$h_j = g_{(h_1, \ldots, h_r)} \circ q_{R,j}$ for $j = 1, \ldots, r$.
For $x \in Q$, let $T_x$ be the set of all images
$g_{(h_1, \ldots, h_r)}(x) \in A$ (with $(S_1, \ldots, S_r)$ fixed),
and $g_{(S_1, \ldots, S_r)}(x) = f(T_x)$. Then
$g_{(S_1, \ldots, S_r)} : Q_R \to A$ is a homomorphism,
and for $x \in q_{R,j}(P)$ we have $T_x$ = $S_x$ (because they are images
of $x$ under restrictions of the same homomorphisms),
whence $g_{(S_1, \ldots, S_r)}(x) = f_{S_j}(x)$.
Thus $f_{S_j} = g_{(S_1, \ldots, S_r)} \circ q_{R,j}$ for $j = 1, \ldots, r$.
Consequently $(f_{S_1}, \ldots, f_{S_r}) 
= (\hat{f}(S_1), \ldots, \hat{f}(S_r))\in R(\Psi A )$.
This shows that $\hat{f}$ is a homomorphism.
\end{proof}

Unlike the case of the arc-graph construction, we are unable to provide
an explicit description of the tree obstructions of~$\Psi A $ in terms
of those of~$A$ for a general right adjoint~$\Psi$. However, in isolated
cases we can do it, as the following example shows.

\begin{exa}
The endofunctor~$\Psi$ on the category of digraphs is defined via the
pattern $\{P, (Q, q_1, q_2)\}$, where $P=\vec P_1$ is the one-arc
path $u\to v$, $Q=\vec P_3$ is the directed path $0\to1\to2\to3$, the
homomorphism $q_1: u\mapsto0,\ v\mapsto 1$, and finally $q_2: u\mapsto
2,\ v\mapsto 3$.

Let $T$ be a tree of algebraic height~$h$ and consider the unique homomorphism~$t$
from~$T$ to the directed path~$\vec P_h$. The arcs of~$T$ are of two
kinds: \emph{blue arcs} $A_b(T)=\{(x,y): t(x)=2k,\ t(y)=2k+1 \text{ for
some integer $k$}\}$ and \emph{red arcs} $A_r(T)=\{(x,y): t(x)=2k+1,\
t(y)=2k+2 \text{ for some integer $k$}\}$. We define two equivalence
relations on the vertices of~$T$: $x\sim_b y$ if the (not necessarily
directed) path from~$x$ to~$y$ in~$T$ has only blue arcs, and $x\sim_r y$
if the path from~$x$ to~$y$ in~$T$ has only red arcs.  Then $T$~has two
$\Psi$-Sproinks, namely $T/{\sim_b}$ and $T/{\sim_r}$ with loops removed.

For a collection~$\T$ of trees, let $\PsiS(\T)$ be the set of all
$\Psi$-Sproinks of the trees contained in~$\T$. We claim that if $\T$~is
a complete set of obstructions for a template~$H$, then $\PsiS(\T)$ is a
complete set of obstructions for~$\Psi H$. To prove it, we follow the idea
of the proofs of Lemma~\ref{lem:sproink} and Theorem~\ref{thm:spr-thm}.

First we prove that $T\to\Psi^{-1}(T/{\sim_b})$. This is not difficult: every
blue arc of~$T$ was contracted to a vertex of~$T/{\sim_b}$ and this vertex
was blown up to an arc in $T\to\Psi^{-1}(T/{\sim_b})$. Thus we can map blue
arcs to the corresponding blown-up arcs. Red arcs of~$T$ are also arcs
of~$T/{\sim_b}$, and hence we can map each red arc to the arc $(1,2)$
of the corresponding copy of~$Q$ in $\Psi^{-1}(T/{\sim_b})$. Clearly such
a mapping is a homomorphism.

Analogously we show that $T\to\Psi^{-1}T/{\sim_r}$.

Finally we want to prove that if $T\to\Psi^{-1}G$, then
either $T/{\sim_b}\to G$ or $T/{\sim_r}\to G$. Suppose that
$f:T\to\Psi^{-1}G$. Then some arcs of~$T$ are mapped by~$f$ to arcs
corresponding to vertices of~$G$ (arcs of copies of~$P$), and others
are mapped to arcs corresponding to arcs of~$G$ (arcs $(1,2)$ of copies
of~$Q$). Let us call the former v-arcs and the latter a-arcs. It follows
from the definition of~$\Psi^{-1}$ that either all blue arcs of~$T$
are v-arcs and all red arcs of~$T$ are a-arcs, or all blue arcs of~$T$
are a-arcs and all red arcs of~$T$ are v-arcs. In the former case
$T/{\sim_b}\to G$, while in the latter case $T/{\sim_r}\to G$.
\end{exa}

It is notable that in the above example each tree obstruction for~$H$
generates finitely many obstructions for~$\Psi H $. This is no accident.

\begin{theorem} \label{thm:fidu}
Let $\Psi$ be a functor generated by a pattern
$\{ P \} \cup \{ (Q_R, q_{R,1}, \ldots, q_{R,a(R)}) : R \in \tau \}$,
where for every $R \in \tau$ and $1 \leq i < j \leq a(R)$,
the image $q_{R,i}(P)$ is vertex-disjoint from $q_{R,j}(P)$. If a
$\sigma$-structure $A$ has finite duality, then the $\tau$-structure
$\Psi A $ also has finite duality.
\end{theorem}

The proof uses the characterisation of structures with finite
duality of~\cite{Tar:FOD}. The \emph{square} of a 
$\sigma$-structure~$B$ is the structure~$B \times B$. 
It contains the \emph{diagonal}
$\Delta_{B \times B} = \{ (b,b) : b \in B\}$.
An element~$a$ of~$B$ is \emph{dominated} by
an element~$b$ of~$B$ if for every $R \in \sigma$, for every~$i$ and
every $(x_1,\ldots,x_{a(R)}) \in R(B)$ with $x_i = a$, we have
$(y_1,\ldots,y_{a(R)}) \in R(B)$ with $y_i = b$ and $y_j = x_j$ for $j \neq i$.
A structure~$B$ \emph{dismantles to} its induced substructure~$C$
if there exists a sequence $x_1,\dots,x_k$ of distinct elements 
of~$B$ such that $B\setminus C = \{x_1,\dots,x_k\}$ and for each 
$1 \leq i \leq k$ the element $x_i$ is dominated in the structure 
induced by $C \cup \{x_i,\dots,x_k\}$. 
The sequence $x_1,\dots,x_k$ is then called a \emph{dismantling
sequence} of~$B$ on~$C$. 

\begin{theorem}[\cite{Tar:FOD}]
A structure has finite duality if and only if it has a retract whose
square dismantles to its diagonal.
\qed
\end{theorem}

\begin{proof}[Proof of Theorem~\ref{thm:fidu}]
Let $A$ be a $\sigma$-structure with finite duality. Without loss of
generality, we assume that $A$ is a core, so that $A$ has no
proper retracts; thus the square of $A$ dismantles to its diagonal.
Let $(x_1,y_1), \ldots, (x_k, y_k)$ be a dismantling sequence
of $A \times A$ on $\Delta_{A \times A}$. Then 
$\Psi(A \times A) \cong \Psi A  \times \Psi A $; we want to prove
that it dismantles to 
$\Delta_{\Psi A  \times \Psi A } \cong \Psi \Delta_{A \times A} $.

For $i \in \{1, \ldots, k\}$, define $X_i$ to be the substructure
of $A \times A$ induced by the set 
$\Delta_{A \times A} \cup \{(x_i,y_i), \ldots, (x_k,y_k)\}$,
and let $X_{k+1} = \Delta_{A \times A}$.
We will show that $\Psi X_i $ can be dismantled to
$\Psi X_{i+1} $, $i = 1, \ldots, k$.

Let $b = (b_1,b_2)$ be an element dominating $a = (x_i,y_i)$
in $X_i$.
Let $f \in \Psi X_i  \setminus \Psi X_{i+1} $, 
and assume that $f = (f_1, f_2): P \to A \times A$.
Then there exists
(at least one) $p_0 \in P$ such that 
$f(p_0) = a$. We define 
$g = (g_1, g_2): P \to A \times A$ by
$g(p_0) = b$ and $g(p) = f(p)$ if $p \neq p_0$. 
Since $b$ dominates $a$, $g$ is a homomorphism,
and obviously $g \in \Psi X_i $. We claim that
$g$ dominates $f$. Indeed, for $R \in \tau$
and $(f_1, \ldots, f_{a(R)}) \in R(\Psi X_i )$
such that $f = f_j$, there exists a homomorphism
$h: Q_R \to X_i$ such that $f = h \circ q_{R,j}$.
Define $h': Q_R \to X_i$ by $h'(q_{R,j}(p_0)) = b$
and $h'(z) = h(z)$ for $z \neq q_{R,j}(p_0)$.
Since $b$ dominates $a = h(q_{R,j}(p_0))$,
the mapping~$h'$ is a homomorphism. By hypothesis, for $\ell \neq j$,
the image $q_{R,\ell}(P)$ is disjoint from $q_{R,j}(P)$,
whence $f_{\ell} = h' \circ q_{R,\ell}$, while
$h' \circ q_{R,j} = g$. Therefore $R(\Psi X_i )$ contains
all the $a(R)$-tuples needed to establish the domination of~$f$
by~$g$.

Let $p_1, p_2, \ldots, p_m$ be an enumeration
of the elements of $P$. We dismantle $\Psi X_i $
to $\Psi X_{i+1} $ by successively removing the functions
$f$ such that $f(p_j) = (x_i,y_i)$ for $j = 1, \ldots, m$.
Proceeding in this way for $i = 1, \ldots, k$,
we get a dismantling of $\Psi A  \times \Psi A  \cong \Psi X_1 $
to $\Psi X_{k+1}  \cong \Delta_{\Psi A  \times \Psi A }$.
Therefore $\Psi A $ has finite duality.
\end{proof}

Perhaps the lack of knowledge of a general construction is natural since
there is no restriction on the pattern
$\{ P \} \cup \{ (Q_R, q_{R,1}, \ldots, q_{R,a(R)}) : R \in \tau \}$.
On the other hand, there are many possible transformations
${\cal T}'$ on a family ${\cal T}$ of tree obstructions,
in the style of $\Sproink({\cal T})$.
Any such transformation gives rise to a complete set
of obstructions to homomorphisms into a structure
$H' = \Pi_{T \in {\cal T}'}D_T$; however in general there
is no way of guaranteeing that such structure $H'$ is finite,
even when ${\cal T}$ is a complete set of obstructions for
a finite structure $H$. 

\section{Concluding comments}

In this paper we tried to shed more light on the structure of
tractable templates with tree duality. Let us
turn our attention one more time to Fig.~\ref{fig:diag}. The grey areas
in the diagram are areas that need a closer look in future research.

Currently we do not know any digraph with a near-unanimity function
and with bounded-height tree duality that could not be generated using right
adjoints and products, starting from digraphs with finite duality;
it is not clear whether any such ``reasonable'' class of
structures with tree duality can be generated from structures
with finite duality with a ``reasonable'' set of adjoint functors.

We have shown here that possession of bounded-height tree duality is
decidable. It is natural to ask what its complexity is; in particular,
whether it is complete for some class of problems.

Equally interesting is the decidability of membership in other
classes depicted in Fig.~\ref{fig:diag}. Tree duality is known to be
decidable~\cite{FedVar:SNP}, but not known to be in PSPACE. Our decision
procedure for bounded-height duality is in PSPACE for graphs with tree
duality; this suggests that checking tree duality may be harder than
checking bounded height of the obstructions.

Furthermore, finite duality is
NP-complete~\cite{Tar:FOD}. The decidability of bounded-tree\-width duality
is unknown, and so is the decidability of a near-unanimity function
(see~\cite{Mar:ExNUF} for a related result).

The properties of near-unanimity functions proved in the proof of
Proposition~\ref{prop:dpc} (i) in the context of digraphs 
and the arc-graph construction, also hold in the context
of general structures and right adjoints. The proofs 
carry over naturally.


\begin{thebibliography}{10}

\bibitem{BarWel:Cat}
M.~Barr and C.~Wells.
\newblock {\em Category Theory for Computing Science}.
\newblock Les Publications CRM, Montr{\'e}al, 3rd edition, 1999.

\bibitem{CohJea:CCL}
D.~Cohen and P.~Jeavons.
\newblock The complexity of constraint languages.
\newblock In F.~Rossi, P.~van Beek, and T.~Walsh, editors, {\em Handbook of
  Constraint Programming}, volume~2 of {\em Foundations of Artificial
  Intelligence}, chapter~8. Elsevier, 2006.

\bibitem{FedVar:SNP}
T.~Feder and M.~Y. Vardi.
\newblock The computational structure of monotone monadic {SNP} and constraint
  satisfaction: {A} study through {D}atalog and group theory.
\newblock {\em SIAM J. Comput.}, 28(1):57--104, 1999.

\bibitem{F:Diss}
J.~Foniok.
\newblock {\em Homomorphisms and Structural Properties of Relational Systems}.
\newblock PhD thesis, Charles University, Prague, 2007.

\bibitem{HelNes:Dicho}
P.~Hell and J.~Ne{\v s}et{\v r}il.
\newblock On the complexity of {$H$}-coloring.
\newblock {\em J. Combin. Theory Ser. B}, 48(1):92--110, 1990.

\bibitem{HelNes:GrH}
P.~Hell and J.~Ne{\v s}et{\v r}il.
\newblock {\em Graphs and Homomorphisms}, volume~28 of {\em Oxford Lecture
  Series in Mathematics and Its Applications}.
\newblock Oxford University Press, 2004.

\bibitem{HelNesZhu:DualPoly}
P.~Hell, J.~Ne{\v s}et{\v r}il, and X.~Zhu.
\newblock Duality and polynomial testing of tree homomorphisms.
\newblock {\em Trans. Amer. Math. Soc.}, 348(4):1281--1297, 1996.

\bibitem{Hod:A-shorter-model}
W.~Hodges.
\newblock {\em A shorter model theory}.
\newblock Cambridge University Press, 1997.

\bibitem{Kom:Phd}
P.~Kom{\'a}rek.
\newblock {\em Good characterisations in the class of oriented graphs}.
\newblock PhD thesis, Czechoslovak Academy of Sciences, Prague, 1987.
\newblock In Czech (Dobr{\'e} charakteristiky ve t{\v r}{\'\i}d{\v e}
  orientovan{\'y}ch graf{\r u}).

\bibitem{Lad:StruPTR}
R.~E. Ladner.
\newblock On the structure of polynomial time reducibility.
\newblock {\em J. Assoc. Comput. Mach.}, 22(1):155--171, 1975.

\bibitem{Tar:FOD}
B.~Larose, C.~Loten, and C.~Tardif.
\newblock A characterisation of first-order constraint satisfaction problems.
\newblock In {\em Proceedings of the 21st IEEE Symposium on Logic in Computer
  Science (LICS'06)}, pages 201--210. IEEE Computer Society, 2006.

\bibitem{LoTar:mosfd}
C.~Loten and C.~Tardif.
\newblock Majority functions on structures with finite duality.
\newblock {\em European J. Combin.}, 29(4):979--986, 2008.

\bibitem{Mac:Cat}
S.~Mac~Lane.
\newblock {\em Categories for the Working Mathematician}, volume~5 of {\em
  Graduate texts in mathematics}.
\newblock Springer-Verlag, New York Berlin Heidelberg, 2nd edition, 1998.

\bibitem{Mar:ExNUF}
M.~Mar{\'o}ti.
\newblock The existence of a near-unanimity term in a finite algebra is
  decidable.
\newblock Manuscript, 2005.

\bibitem{NesTar:Dual}
J.~Ne{\v s}et{\v r}il and C.~Tardif.
\newblock Duality theorems for finite structures (characterising gaps and good
  characterisations).
\newblock {\em J. Combin. Theory Ser. B}, 80(1):80--97, 2000.

\bibitem{NesTar:Short}
J.~Ne{\v s}et{\v r}il and C.~Tardif.
\newblock Short answers to exponentially long questions: Extremal aspects of
  homomorphism duality.
\newblock {\em SIAM J. Discrete Math.}, 19(4):914--920, 2005.

\bibitem{Pul:The-right-adjoints}
A.~Pultr.
\newblock The right adjoints into the categories of relational systems.
\newblock In {\em Reports of the Midwest Category Seminar, IV}, volume 137 of
  {\em Lecture Notes in Mathematics}, pages 100--113, Berlin, 1970. Springer.

\end{thebibliography}
\end{document}